\documentclass{article}
\usepackage{graphicx} 
\usepackage[USenglish]{babel}
\usepackage{bbm}
\usepackage{amsmath}
\usepackage{amsfonts}
\usepackage{amssymb}
\usepackage{amsthm}
\usepackage{bussproofs}
\usepackage{enumerate}
\usepackage{enumitem}
\usepackage{mathtools}
\usepackage[hidelinks]{hyperref}
\usepackage{scrextend}
\usepackage{tikz}
\usepackage[autostyle=false, style=english]{csquotes}
\usepackage{commath}

\makeatletter
\newcommand*{\rn}[1]{%
  \expandafter\@rn\csname c@#1\endcsname%
}

\newcommand*{\@rn}[1]{%
  $\ifcase#1\or(i)\or(ii)\or(iii)\or(iv)\or(v)\or(vi)\or(vii)\or(viii)\or(ix)\or(x)%
    \else\@ctrerr\fi$%
}
\AddEnumerateCounter{\rn}{\@rn}{53.13}
\makeatother

  \newbox\gnBoxA
\newdimen\gnCornerHgt
\setbox\gnBoxA=\hbox{$\ulcorner$}
\global\gnCornerHgt=\ht\gnBoxA
\newdimen\gnArgHgt
\def\Godelnum #1{%
\setbox\gnBoxA=\hbox{$#1$}%
\gnArgHgt=\ht\gnBoxA%
\ifnum     \gnArgHgt<\gnCornerHgt \gnArgHgt=0pt%
\else \advance \gnArgHgt by -\gnCornerHgt%
\fi \raise\gnArgHgt\hbox{$\ulcorner$} \box\gnBoxA %
\raise\gnArgHgt\hbox{$\urcorner$}}

\makeatletter
\newcommand{\pushright}[1]{\ifmeasuring@#1\else\omit\hfill$\displaystyle#1$\fi\ignorespaces}
\newcommand{\pushleft}[1]{\ifmeasuring@#1\else\omit$\displaystyle#1$\hfill\fi\ignorespaces}
\makeatother

\newcommand{\PP}{\mathbb{P}}

\newcommand{\1}{\mathbbm{1}}
\newcommand{\forces}{\Vdash}
\newcommand{\res}{\upharpoonright}

\newcommand{\dotieconcat}[2]{\text{\raisebox{.8ex}{$\smallfrown$}}}

\newcommand{\QQ}{\mathbb{Q}}

\newcommand{\kapman}{\mathbb{M}_{\kappa}}
\newcommand{\jkapman}{\mathbb{M}_{j(\kappa)}}
\newcommand{\mantle}{\mathbb{M}}

\newcommand{\zerosharp}{0^\#}

\newcommand{\cof}{\operatorname{cof}}

\newcommand{\ran}{\mathrm{ran}}
\newcommand{\ZFC}{\mathrm{ZFC}}
\newcommand{\ZF}{\mathrm{ZF}}

\newcommand{\GCH}{\mathrm{GCH}}

\newtheoremstyle{nopoint}
  {}{}{\itshape}{}{\bfseries}{}{5pt}{}

\theoremstyle{plain}
\newtheorem{thm}{Theorem}[section]
\newtheorem{prop}[thm]{Proposition}
\newtheorem{lemm}[thm]{Lemma}
\newtheorem{fact}[thm]{Fact}

\newtheorem{cor}[thm]{Corollary}
\newtheorem{claim}[thm]{Claim}

\theoremstyle{definition}
\newtheorem{defn}[thm]{Definition}

\newtheorem{rem}[thm]{Remark}
\newtheorem{que}[thm]{Question}
\newtheorem*{que*}{Question}

\theoremstyle{nopoint}
\newtheorem*{lemm*}{Lemma}
\newtheorem*{thm*}{Theorem}
\newtheorem*{rem*}{Remark}
\usepackage{mathabx}
\usepackage{tikz-cd}
\usepackage{tikz}\usetikzlibrary{arrows}

\newcommand{\Col}{\mathrm{Col}}
\newcommand{\Add}{\mathrm{Add}}

\newcommand{\AC}{\mathrm{AC}}
\newcommand{\crit}{\mathrm{crit}}

\newcommand{\supp}{\mathrm{supp}}

\newcommand{\Ord}{\mathrm{Ord}}

\numberwithin{thm}{section}

\theoremstyle{plain}

\theoremstyle{definition}
\newtheorem*{defn*}{Definition}

\newcommand{\oneman}{\mantle_{\omega_1}}
\newcommand{\cohman}{\mantle_{\mathsf{C}}}
\newcommand{\succman}{\mantle_{\kappa^+}}

\newcommand{\PPbar}{\underbar{$\PP$}}
\newcommand{\Cbar}{\underbar{$\dot C$}}

\newcommand{\AxA}{\mathrm{AA}}
\newcommand{\SAxA}{\mathrm{SAA}}

\usetikzlibrary{fit}

\title{The Axiom of Choice in the $\kappa$-Mantle}
\author{Andreas Lietz\footnote{Institut f\"ur Mathematische Logik und Grundlagenforschung, Universit\"at M\"unster, Einsteinstrasse 62, 48149 M\"unster, FRG. }\hspace{5pt}\footnote{Current address: Institut für Diskrete Mathematik und Geometrie, TU Wien, Wiedner Hauptstrasse 8-10/104, 1040 Wien, AT\\
This paper is part of the authors PhD thesis.}}
\date{September 2023}

\begin{document}

\maketitle

\begin{abstract}
    Usuba has asked whether the $\kappa$-mantle, the intersection of all grounds that extend to $V$ via a forcing of size ${<}\kappa$, is always a model of $\ZFC$. We give a negative answers by constructing counterexamples where $\kappa$ is a Mahlo cardinal, $\kappa=\omega_1$ and where $\kappa$ is the successor of a regular uncountable cardinal.
\end{abstract}

\section{Introduction}

Set-Theoretic Geology is the study of the structure of grounds\index{Ground}, that is inner models of $\ZFC$ that extend to $V$ via forcing, and associated concepts. Motivated by the hope to uncover canonical structure hidden underneath generic sets, the mantle was born.

\begin{defn}
The mantle\index{Mantle}, denoted $\mantle$, is the intersection of all grounds.
\end{defn}

This definition only makes sense due to the uniform definability of grounds.

\begin{fact}
There is a first order $\in$-formula $\varphi(x, y)$ such that 
$$W_r=\{x\vert \varphi(x, r)\}$$
defines a ground for all $r\in V$ and all grounds are of this form. Moreover, if $\kappa$ is a cardinal and $W$ extends to $V$ via a forcing of size ${<}\kappa$ then there is $r\in V_\kappa$ with $W=W_r$.
\end{fact}
This was proven independently by Woodin \cite{woodinmultiverse} \cite{woodinrecentdevelopments}, Laver \cite{laverextension} and was later strengthened by Hamkins, see \cite{fhrstg}.

This allows us to quantify freely over grounds as we will frequently do.
\smallskip\\
It was quickly realized that every model of $\ZFC$ is the mantle of another model of $\ZFC$, see \cite{fhrstg}, which eradicated any chance of finding nontrivial structure in the mantle. However, the converse question remained open for some while, namely whether the mantle is provably a model of $\ZFC$. This tough nut was cracked by Toshimichi Usuba.
\begin{fact}[Usuba,\cite{usuddg}]
 The mantle is always a model of $\ZFC$.   
\end{fact}

Thereby the mantle was established as a well behaved canonical object in the theory of forcing. Fuchs-Hamkins-Reitz \cite{fhrstg} suggested to study restricted forms of the mantle.

\begin{defn}
Let $\Gamma$ be a class\footnote{In this case, we think of $\Gamma$ as a definition, possibly with ordinal parameters, so that $\Gamma$ can be evaluated grounds of $V$.} $\Gamma$ of forcings.

\begin{enumerate}[label=\rn*]
\item A $\Gamma$\textit{-ground}\index{Ground!Gamma@$\Gamma$-} is a ground $W$ that extends to $V$ via a forcing $\mathbb P\in\Gamma^W$.
\item The $\Gamma$\textit{-mantle}\index{Mantle!Gamma@$\Gamma$-} $\mantle_\Gamma$ is the intersection of all $\Gamma$-grounds.
\item We say that the $\Gamma$-grounds are \textit{downwards directed}\index{Ground!Gamma@$\Gamma$!downwards directed} if for any two $\Gamma$-grounds $W_0, W_1$ there is a $\Gamma$-ground $W_\ast\subseteq W_0, W_1$.
\item We say that the $\Gamma$-grounds are downwards set-directed\index{Ground!Gamma@$\Gamma$!downwards set-directed} if for any set-indexed collection of $\Gamma$-grounds $\langle W_r\mid r\in X\rangle$ there is a $\Gamma$-ground $W_\ast$ contained in all $W_r$ for $r\in X$.
\item We say that $\Gamma$ is \textit{ground absolute} if the $\Gamma$-grounds of a $\Gamma$-ground $W$ are exactly those common grounds of $V$ and $W$ that are $\Gamma$-grounds from the perspective of $V$, i.e. being a $\Gamma$-ground is absolute between $V$ and all $\Gamma$-grounds.
\end{enumerate}
\end{defn}

\begin{rem}
Note that if $\Gamma$ is provably (in $\ZFC$) closed under quotients and two-step iterations then $\Gamma$ is ground absolute.
\end{rem}
Fuchs-Hamkins-Reitz \cite{fhrstg} have shown abstractly that if $\Gamma$ is ground absolute and has directed grounds then $\mantle_\Gamma\models \ZF$. To prove $\mantle\models\AC$ they seemingly need the stronger assumption that the $\Gamma$-grounds are downwards set-directed, the argument is as follows: Suppose $X\in\mantle$ is not wellordered in $\mantle$. Then for every wellorder $\prec$ of $X$, we choose $W_\prec$ a $\Gamma$-ground from which $\prec$ is missing. By downwards set directedness, there is a $\Gamma$-ground $W$ contained in all such grounds $W_\prec$, but then $X\in W$ is not wellordered in $W$ either, contradiction. The main result of this part shows that indeed simple downwards directedness does not suffice to prove choice in $\mantle_\Gamma$ in general.\\

We will be interested in $\mantle_\Gamma$ for $\Gamma$ the class of all forcings of size ${<}\kappa$, where $\kappa$ is some given cardinal. In this case, we denote the $\Gamma$-mantle by $\kapman$ and call it the $\kappa$\textit{-mantle}\index{Mantle!kappa@$\kappa$-}. The associated grounds are the $\kappa$\textit{-grounds}\index{Ground!kappa@$\kappa$-}. The interest of the $\kappa$-mantle arose in different contexts. 

The following is known:

\begin{fact}[Usuba, \cite{usuext}]\label{stronglimitzffact}
If $\kappa$ is a strong limit then $\kapman\models \ZF$.
\end{fact}

Usuba proved this by showing that the $\kappa$-grounds are directed in this case. Usuba subsequently asked:

\begin{que}[Usuba, \cite{usuext}]
Is $\kapman$ always a model of $\ZFC$?
\end{que}

We will answer this question in the negative by providing counterexamples for three different types of cardinals $\kappa$.\\

We also mention that Fuchs-Hamkins-Reitz demonstrated that $\mantle_\Gamma$ can fail to be a model of choice for a different class of forcings, $\Gamma=\{\sigma\text{-closed forcings}\}$.

\begin{fact}[Fuchs-Hamkins-Reitz, \cite{fhrstg}]
If $\Gamma$ is the class of all $\sigma$-closed forcings it is consistent that $\mantle_\Gamma\models \ZF\wedge\neg \mathrm{AC}$.
\end{fact}


It turns out that there is an interesting tension between large cardinal properties of $\kappa$ and the failure of choice in $\kapman$. On the one side, Usuba has shown:

\begin{fact}[Usuba, \cite{usuext}]\label{usubaext}
If $\kappa$ is extendible then $\kapman=\mantle$. In particular $\kapman$ is a model of $\ZFC$.
\end{fact}

Indeed, this result was the initial motivation of investigating the $\kappa$-mantle. Sargsyan-Schindler \cite{ssvar} showed that a similar situation arises in the least iterable inner model with a strong cardinal above a Woodin cardinal for $\kappa$ the unique strong cardinal in this universe. See also \cite{varsovianii} and \cite{varsovianomega} for further results in this direction.\\
On another note, Schindler has proved the following.
\begin{fact}[Schindler, \cite{schnotes}]\label{measurablethm}
If $\kappa$ is measurable then $\kapman\models \ZFC$.
\end{fact}

The big difference to Fact \ref{usubaext} is that the existence of a measurable is consistent with the failure of the Bedrock Axiom\footnote{The Bedrock Axiom states that the universe has a minimal ground, which turns out to be equivalent to ``$\mantle$ is a ground".}. Particularly, we might have $\kapman\neq\mantle$ for $\kappa$ measurable.\\
If we go even lower in the large cardinal hierarchy then even less choice principles seem to be provable in the corresponding mantle. The relevant results here are due to Farmer Schlutzenberg.

\begin{fact}[Schlutzenberg, \cite{farmermantles}]
Suppose that $\kappa$ is weakly compact. Then 
\begin{enumerate}[label=$(\roman*)$]
\item $\kapman\models\kappa\text{-}\mathrm{DC}$ and
\item for every $A\in H_{\kappa^+}\cap \kapman$,
$$\kapman\models``A\in H_{\kappa^+}\text{ is wellorderable}".$$
\end{enumerate}
\end{fact}

\begin{defn}
Suppose $\alpha$ is an ordinal and $X$ is a set. $({<}\alpha, X)$-choice\index{(<alpha X) choice@$({<}\alpha, X)$-choice} holds if for any $\beta<\alpha$ and any sequence $\vec x\coloneqq\langle x_\gamma\mid\gamma<\beta\rangle$ of nonempty elements of $X$ there is a choice sequence for $\vec x$, that is a sequenece $\langle y_\gamma\mid\gamma<\beta\rangle$ with $y_\gamma\in x_\gamma$ for all $\gamma<\beta$.
\end{defn}

\begin{fact}[Schlutzenberg, \cite{farmermantles}]\label{farmerinaccessiblefact}
Suppose $\kappa$ is inaccessible. Then we have 
\begin{enumerate}[label=$(\roman*)$]
    \item $V_\kappa\cap\kapman\models\ZFC$ and
    \item $\kapman\models ``({<}\kappa, H_{\kappa^+})\text{-choice}"$.
\end{enumerate}
\end{fact}

\subsection*{Acknowledgments}
The author thanks Ralf Schindler and Farmer Schlutzenberg for valuable discussions on the present topic. Moreover, the author thanks Ralf Schindler for his support and guidance during his work on this project.\\
Funded by the Deutsche Forschungsgemeinschaft (DFG, German Research
Foundation) under Germany’s Excellence Strategy EXC 2044 390685587, Mathematics M\"unster:
Dynamics - Geometry - Structure.

\section{Overview}
In Section \ref{kappamahlosection}, we will argue that ``$\kappa$ is measurable" cannot be replaced by ``$\kappa$ is Mahlo" in Fact \ref{measurablethm}, as wells as that $({<}\kappa, H_{\kappa^+})$-choice cannot be strengthened to $({<}\kappa+1, H_{\kappa^+})$-choice in Fact \ref{farmerinaccessiblefact}.

\begin{thm}\label{nochoice} 
If $\ZFC$ is consistent with the existence of a Mahlo cardinal, then it is consistent with $\ZFC$  that there is a Mahlo cardinal $\kappa$ so that $\kapman$ fails to satisfy the axiom of choice. In fact we may have
$$\kapman\models``({<}\kappa+1,H_{\kappa^+})\text{-choice fails}".$$
\end{thm}

In Section \ref{omega1mantlesection}, we will investigate the $\kappa$-mantle for $\kappa=\omega_1$, as well as the $\Gamma$-mantle where $\Gamma=\{\text{Cohen forcing}\}$, denoted by $\cohman$. We will first proof that these mantles are always models of $\ZF$ and will go on to provide a result analogous to Theorem \ref{nochoice}.

\begin{thm}
It is consistent relative to a Mahlo cardinal that both $\oneman$ and $\cohman$ fail to satisfy the axiom of choice. 
\end{thm}

In Section \ref{successorcardinalsection}, we will generalize this to any successor of a regular cardinal.

\begin{thm}
Suppose that 
\begin{enumerate}[label=$(\roman*)$]
    \item $\GCH$ holds,
    \item the Ground Axiom\footnote{The Ground Axiom\index{Ground Axiom} states that there is no nontrivial ground. See \cite{reitga} for more information on this axiom.} holds and
    \item $\kappa$ is a regular uncountable cardinal.
\end{enumerate}
Then there is a cardinal preserving generic extension in which the $\kappa^+$-mantle fails to satisfy the axiom of choice.
\end{thm}

In this case however, it is not known if the $\kappa^+$-mantle is a model of $\ZF$ in general. The proof of all these three theorems follows a similar pattern, though the details differ from case to case and it seems that we cannot employ a fully unified approach.

\section{The Axiom of Choice May Fail in $\kapman$}

\subsection{The case ``$\kappa$ is Mahlo"}\label{kappamahlosection}

Here, we will construct a model where the $\kappa$-mantle for a Mahlo cardinal $\kappa$ does not satisfy the axiom of choice. We will start with $L$ and assume that $\kappa$ is the least Mahlo there. The final model will be a forcing extension of $L$ by
$$\PP=\prod_{\lambda\in I\cap\kappa}^{\text{${<}\kappa$-support}} \Add(\lambda, 1)$$
where $I$ is the class of all inaccessible cardinals. We define $\PP$ to be a product forcing and not an iteration (in the usual sense), as we want to generate many $\kappa$-grounds. Let $G$ be $\PP$-generic over $L$. We will show that $\kappa$ is still Mahlo in $L[G]$ and that $\kapman^{L[G]}$ does not satisfy the axiom of choice. We remark that, would we start with a model in which $\kappa$ is measurable, $\PP$ would provably force $\kappa$ to not be measurable.
\medskip\\
First, let's fix notation. For $\lambda<\kappa$, we may factor $\PP$ as $\PP_{\leq\lambda}\times\PP_{>\lambda}$ where in each case we only take a product over all $\gamma\in I\cap\kappa$ with $\gamma\leq\lambda$ and $\gamma>\lambda$ respectively. Observe that $\PP_{>\lambda}$ is a ${<}\kappa$-support product while $\PP_{\leq\lambda}$ is a full support product. We also factor $G$ as $G_{\leq\lambda}\times G_{>\lambda}$ accordingly. For $\lambda\in I\cap\kappa$ we denote the generic for $\Add(\lambda, 1)^L$ induced by $G$ as $g_\lambda$. In addition to this, for $\alpha\leq\kappa$ we denote the $\alpha$-th inaccessible cardinal by $I_\alpha$. \\
For $\alpha<\kappa$ let $E_\alpha\colon \kappa\rightarrow 2$ be the function induced by $g_{I_\alpha}$.  It will be convenient to think of $G$ as a $\kappa\times\kappa$-matrix $M$ which arises by stacking the maps $(E_\alpha)_{\alpha<\kappa}$ on top of each other, starting with $E_{I_0}$ and proceeding downwards, and then filling up with $0$'s to produce rows of equal length $\kappa$. Let us write 
$$e_{\alpha, \beta}=\begin{cases}
E_\alpha(\beta) & \text{if }\beta< I_\alpha\\
0 & \text{else.}\\
\end{cases}
$$

The $(e_{\alpha, \beta})_{\alpha, \beta<\kappa}$ are the entries of $M$:

\begin{tikzpicture}
\matrix (m)[matrix of math nodes, left delimiter=(,right delimiter=), anchor = west, inner sep=6pt] at (0, 0){
e_{0, 0} & e_{0, 1} & e_{0, 2} & \cdots & 0 & \cdots & 0  & \cdots & 0 & \cdots \\
e_{1, 0} & e_{1, 1} & e_{1, 2} & \cdots & e_{1, I_0} & \cdots & 0 & \cdots & 0 & \cdots\\
\vdots   & \vdots   & \vdots   & \ddots & \vdots          & \ddots & 0 & \cdots & 0 & \cdots\\
e_{\alpha, 0} & e_{\alpha, 1} & e_{\alpha, 2} & \cdots & e_{\alpha, I_0} & \cdots & e_{\alpha, I_1} & \cdots & 0 & \cdots \\
\vdots & \vdots & \vdots & \ddots & \vdots & \vdots & \vdots & \ddots & 0 & \cdots \\
};
\node[red,fit=(m-2-1)(m-2-10),inner sep=1pt, label={[xshift = 0.15cm, text=red]left:$r_1$}, draw, rounded corners]{};
\node[blue,fit=(m-1-3)(m-5-3),inner sep=1pt,label={[text=blue]above:$c_2$}, draw, rounded corners]{};
\node[teal,fit=(m-4-1)(m-5-10),inner sep=1pt,label={[text=teal]below:$M_{\geq\alpha}$}, draw, rounded corners]{};
\node(M) at (-.8, 0) {$M=$};
\end{tikzpicture}

We will give the $\alpha$-th row of $M$ the name $r_\alpha$ and we denote the $\beta$-th column of $M$ by $c_\beta$. One trivial but key observation is that $r_\alpha$ carries the same information as $g_{I_\alpha}$.

We will be frequently interested in the matrix $M$ with its first $\alpha$ rows deleted for some $\alpha<\kappa$, so we will give this matrix the name $M_{\geq\alpha}$. Note that $M_{\geq\alpha}$ corresponds to the generic $G_{\geq I_\alpha}$. Finally observe that we may think of conditions in $\mathbb P$ as partial matrices that approximate such a matrix $M$ in the sense that they already have the trivial $0$'s in the upper right corner, in any row $\alpha<\kappa$ they have information for ${<}I_\alpha$ many $\beta<I_\alpha$ on whether $e_{\alpha, \beta}$ is $0$ or $1$ and they contain non-trivial information in less than $\kappa$-many rows.

\begin{lemm}\label{sameinaccessibleslemm}
$L$ and $L[G]$ have the same inaccessibles.
\end{lemm}

\begin{proof}
First, we show that all limit cardinals of $L$ are limit cardinals in $L[G]$. It is enough to prove that all double successors $\delta^{++}$ are preserved. This is obvious for $\delta\geq\kappa$ as $\PP$ has size $\kappa$. For $\delta<\kappa$, $\PP_{>\delta}$ is ${\leq}\delta^{++}$-closed so that all cardinals $\leq\delta^{++}$ are preserved in $L[G_{>\delta}]$. Furthermore, $\PP_{<\delta}$ has size at most $\delta^{+}$ in $L[G_{>\delta}]$ by $\GCH$ in $L$.  Hence $\delta^{++}$ is still a cardinal in $L[G]$.
\smallskip\\
Now we have to argue that all $\lambda\in I$ remain regular. Again, this is clear if $\lambda>\kappa$. On the other hand, assume $\delta\coloneqq \cof(\lambda)^{L[G]}<\lambda$. As $\PP_{>\delta}$ is ${\leq}\delta$-closed, $\lambda$ is still regular in $L[G_{>\delta}]$. Hence, a witness to $\cof(\lambda)=\delta$ must be added in the extension of $L[G_{{>}\delta}]$ by $\PP_{\leq\delta}$. But this forcing has size $<\lambda$ in $L[G_{{>}\delta}]$ and thus could not have added such a sequence.
\end{proof}

In fact, $\PP$ does not collapse any cardinals (if $V=L$), but some more work is required to prove this. This is, however, not important for our purposes. Next, we aim to show that $\kappa$ remains Mahlo in $L[G]$. 

To prove this, it is convenient to introduce a generalization of Axiom A.

\begin{defn}\label{generalaa}
For $\kappa$ an ordinal, $\lambda$ a cardinal we say that a forcing $\QQ$ satisfies Axiom A$(\kappa, \lambda)$,\index{Axiom A@Axiom A$(\kappa,\lambda)$} abbreviated by $\mathrm{AA}(\kappa, \lambda)$, if there is a sequence $\langle\leq_\alpha\mid\alpha<\kappa\rangle$ of partial orders on $\QQ$ so that
\begin{enumerate}[label=$(\mathrm{AA}.\roman*)$]
\item\label{AAcond1} $\forall \alpha\leq \beta<\kappa\ \leq_{\beta}\subseteq\leq_\alpha\subseteq\leq_\QQ$,
\item\label{AAcond2} for all antichains $A$ in $\QQ$, $\alpha<\kappa$ and $p\in\QQ$ there is $q\leq_\alpha p$ so that $\vert\{a\in A\mid a\| q\}\vert<\lambda$ and
\item\label{AAcond3} for all $\beta<\kappa$ if $\vec p=\langle p_\alpha\mid\alpha<\beta\rangle$ satisfies $p_{\gamma}\leq_\alpha p_\alpha$ for all $\alpha<\gamma<\beta$ then there is a fusion $p_\beta$ of $\vec p$, that is $p_\beta\leq_\alpha p_\alpha$ for all $\alpha<\beta$. 
\end{enumerate}
\end{defn}

\begin{rem}
The usual Axiom A is thus Axiom A$(\omega+1, \omega_1)$.
\end{rem}

\begin{prop}\label{axastatpresprop}
Suppose $\lambda$ is regular uncountable cardinal and $\QQ$ satisfies $\AxA(\lambda,\lambda)$. Then $\QQ$ preserves stationary subsets of $\lambda$.
\end{prop}

\begin{proof}
Suppose $S\subseteq\lambda$ is stationary, $\dot C$ is a $\QQ$-name for a club in $\lambda$ and $p\in\PP$. We will imitate the standard proof that a ${<}\kappa$-closed forcing preserves stationary sets. Let $\langle \leq_\alpha\mid\alpha<\lambda\rangle$ witness that $\QQ$ satisfies $\AxA(\lambda,\lambda)$. 
\begin{claim}
For any $q\in\QQ,\alpha<\lambda$ there is $r\leq_\alpha q$ and some $\alpha<\gamma<\lambda$ with $q\forces \check\gamma\in\dot C$.
\end{claim}
\begin{proof}
Construct a  sequence $\langle q_\alpha\mid\alpha<\omega\rangle$ of conditions in $\QQ$ and an ascending sequence $\langle \gamma_n\mid n<\omega\rangle$ of ordinals with 
\begin{enumerate}[label=$(\roman*)$]
    \item $q_0=q$, $\gamma_0=\alpha$,
    \item $q_{n+1}\leq_{\alpha+n} q_n$ for all $n<\omega$ and
    \item $q_{n+1}\forces``\dot C\cap (\check \gamma_n,\check\gamma_{n+1})\neq\emptyset$
\end{enumerate}
for all $n<\omega$. The construction is immediate using that $\lambda$ is regular uncountable and \ref{AAcond3}. Then by \ref{AAcond2}, there is $q_\ast\leq_\alpha q$ which is below all $q_n$, $n<\omega$. It follows that 
$$q_\ast\forces\check\gamma_\ast\in\dot C$$
where $\gamma_\ast=\sup_{n<\omega}\gamma_n$.
\end{proof}
Suppose toward a contradiction that $p\forces\dot C\cap\check S=\emptyset$. By the claim above, we can build sequences $\langle p_\alpha\mid\alpha<\lambda\rangle$ of conditions in $\QQ$ and an increasing sequence $\langle \gamma_\alpha\mid\alpha<\lambda\rangle$ of ordinals below $\lambda$ so that
\begin{enumerate}[label=$(\roman*)$]
    \item $p_0=p$, 
    \item $p_\beta\leq_\alpha p_\alpha$ for all $\alpha\leq\beta<\lambda$ and
    \item $p_{\alpha+1}\forces\check\gamma_\alpha\in\dot C$ for all $\alpha<\lambda$.
\end{enumerate}
Let $D$ be the set of all limit points ${<}\lambda$ of $\{\gamma_{\alpha}\mid\alpha<\lambda\}$. For any $\alpha<\lambda$, we have 
$$p_{\alpha+1}\forces \check D\cap \gamma_\alpha\subseteq \dot C$$
which shows that $D\cap S=\emptyset$, contradiction.
\end{proof}

\begin{lemm}\label{axamahlocaslemm}
$\PP$ satisfies $\AxA(\kappa,\kappa)$.
\end{lemm}

\begin{proof}
For $\gamma <\kappa$ define $\leq_\gamma$ by $r\leq_\gamma q$ if $r\leq q$ and $r\res\gamma=q\res\gamma$ for $q, r\in\PP$.  We will only show that \ref{AAcond2} holds. So let $p\in\PP$, $\gamma<\kappa$ and $A\subseteq\PP$ a maximal antichain. Let $\langle q_\alpha\vert \alpha<\delta\rangle$ be an enumeration of all conditions in $\PP_{\leq gamma}$ below $p\res\gamma+1$ with $\delta=\vert \PP_{\leq\gamma}\vert$. We construct a $\leq_\gamma$-descending sequence $\langle p_\alpha\vert \alpha\leq\delta\rangle$ of conditions in $\PP$ starting with $p_0=p$ as follows: If $\alpha\leq\delta$ then choose some $\leq_\gamma$-bound of $\langle p_\beta\mid\beta<\alpha\rangle$. This is possible as $\PP_{>\gamma}$ is ${\leq}\delta$-closed, as the next forcing only appears at the next inaccessible. Moreover, if possible and $\alpha<\delta$ make sure that 
$$q_\alpha^\frown p_{\alpha}\res(\gamma, \kappa)$$
is below a condition in $A$. This completes the construction. Set $q\coloneqq q_\kappa$, we will show that $q$ is compatible with at most $\delta$-many elements of $A$. Toward this goal, suppose $a\in A$ and $q$ is compatible with $a$. We may find some $\alpha<\kappa$ so that $a\res\gamma+1=q_\alpha$. It follows that we must have succeeded in the construction of $p_\alpha$ with the additional demand that 
$$q_\alpha^{\frown}p_\alpha\res(\gamma,\kappa)$$
is below a condition in $A$, but this can only be true for $a$. We have shown that for any $a\in A$ compatible with $q$ there is $\alpha<\delta$ with $q_\alpha^{\frown}q\res(\gamma,\kappa)\leq a$ and note that no single $\alpha$ can witness this for more than one element \mbox{of $A$}.
\end{proof}

\begin{cor}
$\kappa$ is Mahlo in $L[G]$.
\end{cor}
\begin{proof}
This follows immediately from Lemma \ref{sameinaccessibleslemm}, Lemma \ref{axamahlocaslemm} and Proposition \ref{axastatpresprop}.
\end{proof}

Next, we aim to find an easier description of $\kapman^{L[G]}$. Recall the $\lambda$-approximation property introduced by Hamkins \cite{hamext}:

\begin{defn}
Let $W\subseteq V$ be an inner model, $\lambda$ an infinite cardinal.
\begin{enumerate}[label=\rn*]
\item For $x\in V$, a $\lambda$-approximation\index{lambda approximation@$\lambda$-approximation} of $x$ by $W$ is of the form $x\cap y$ where $y\in W$ is of size ${\leq}\lambda$.
\item $W\subseteq V$ satisfies the $\lambda$-approximation property\index{lambda approximation property@$\lambda$-approximation property} if whenever $x\in V$ and all $\lambda$-approximations of $x$ by $W$ are in $W$, then $x\in W$.
\end{enumerate}
\end{defn}

 All $\kappa$-grounds satisfy the $\kappa$-approximation property (cf. \cite{fhrstg}). 

\begin{lemm}\label{nicedescrip} 
$\kapman^{L[G]}=\bigcap_{\lambda\in I\cap\kappa} L[G_{>\lambda}]$.
\end{lemm}

\begin{proof}

Suppose $W$ is a $\kappa$-ground of $L[G]$. It is enough to find $\lambda\in I\cap\kappa$ such that $L[G_{>\lambda}]\subseteq W$. Clearly, $\PP\in L\subseteq W$.  As $\kappa$ is a limit of inaccessibles, we may take some $\lambda<\kappa$ inaccessible so that $W$ is a $\lambda$-ground. Thus $W\subseteq L[G]$ satisfies the $\lambda$-approximation property. We will show $G_{>\lambda}\in W$ (even $G_{\geq\lambda}\in W$). Find $\alpha$ with $\lambda=I_\alpha$, it is thus enough to show $M_{\geq\alpha}\in W$. To any $\lambda$-approximation $M_{\geq\alpha}\cap a$ of $M_{\geq\alpha}$ by $W$ corresponds some $a^\prime\subseteq\kappa\setminus\alpha\times\kappa$, $a^\prime\in W$ of size ${<}\lambda$ so that 
$$M_{\geq\alpha}\cap a=M_{\geq\alpha}\res a^\prime \coloneqq \langle e_{\gamma, \beta}\mid (\gamma,\beta)\in a^\prime\rangle.$$
We will show that all such restrictions of $M_{\geq\alpha}$ are in $W$. So let $a\in W$, $a\subseteq \kappa\setminus \alpha\times\kappa$, $\vert a\vert<\lambda$. As $\zerosharp$ does not exist in $W$, there is $b\in L$, $b\subseteq\kappa\setminus\alpha\times\kappa$ of size $<\lambda$ with $a\subseteq b$. For all $\alpha\leq\gamma<\kappa$, the set of $\beta< I_\gamma$ with $(\gamma,\beta)\in b$ is bounded in $I_\gamma$. As described earlier, we may think of conditions in $\PP$ as partial $\kappa\times\kappa$ matrices. With this in mind, the conditions $p\in\PP$ that contain information on the entry $e_{\gamma, \beta}$ for all $(\gamma, \beta)\in b$ form a dense set of $\PP$. Thus $M\res b=\langle e_{\gamma, \beta}\mid(\gamma, \beta)\in b\rangle$ is essentially a condition $p\in\PP\subseteq W$ and hence $M\res a= (M\res b)\res a\in W$. As $W\subseteq L[G]$ satisfies the $\lambda$-approximation property, we have $M_{\geq\alpha}\in W$.

\end{proof}

\begin{rem}
The above argument shows that for any $\lambda\in I\cap\kappa$
$$\kapman^{L[G_{>\lambda}]}=\kapman^{L[G]}.$$ 
In fact, whenever $\delta$ is a strong limit, the $\delta$-mantle is always absolute to any $\delta$-ground.
The use of Jensen's covering lemma in the above argument is not essential, in fact a model in which the $\kappa$-mantle does not satisfy choice for $\kappa$ Mahlo can be analogously constructed in the presence of $0^\sharp$. However, the absence of $0^\sharp$ simplifies the proof.
\end{rem}

We will later show that $\mathcal{P}(\kappa)^{\kapman^{L[G]}}$ does not admit a wellorder in $\kapman^{L[G]}$. First, we analyze which subsets of $\kappa$ $\kapman^{L[G]}$ knows of. We call $a\subseteq\kappa$ \textit{fresh} if $a\cap \lambda\in L$ for all $\lambda<\kappa$.

\begin{prop}
The subsets of $\kappa$ in $\kapman^{L[G]}$ are exactly the fresh subsets of $\kappa$ in $L[G]$.
\end{prop}

\begin{proof}
First suppose $a\subseteq\kappa$, $a\in\kapman^{L[G]}$. If $\lambda<\kappa$ then $a\in L[G_{>\lambda}]$. As $\PP_{>\lambda}$ is ${\leq}\lambda$-closed in $L$, $a\cap\lambda\in L$.\\
For the other direction assume $a\in L[G]$ is a fresh subset of $\kappa$ and assume $W$ is a $\kappa$-ground of $L[G]$. There is $\lambda<\kappa$ so that $W\subseteq L[G]$ satisfies the $\lambda$-approximation property. As $a$ is fresh, all the $\lambda$-approximations of $a$ in $W$ are in $W$. Thus $a\in W$.
\end{proof}

The columns $c_\beta$, $\beta<\kappa$, of $M$ are the fresh subsets of $\kappa$ relevant to our argument.

\begin{prop}
All $c_\beta$, $\beta<\kappa$, are $\Add(\kappa, 1)$-generic over $L$.
\end{prop} 

\begin{proof}
The map $\pi\colon \PP\rightarrow \Add(\kappa,1)$ that maps $p\in\PP$ to the information that $p$ has on $c_\beta$ is well-defined as $\PP$ is a bounded support iteration of length $\kappa$. Clearly, $\pi$ is a projection.
\end{proof}

This is exactly the reason we chose bounded support in the definition \mbox{of $\PP$}.

We are now in good shape to complete the argument.

\begin{thm}
$({<}\kappa+1, H_{\kappa^+})$-choice fails in $\kapman^{L[G]}$.
\end{thm}

\begin{proof}
Note that any generic for $\Add(\kappa, 1)^L$ is the characteristic function of a fresh subset of $\kappa$ so that $c_\beta\in \kapman^{L[G]}$ for any $\beta<\kappa$. Of course, the sequence $\langle c_\beta\vert\beta<\kappa\rangle$ is not in $\kapman^{L[G]}$, as one can compute the whole matrix $M$ (and thus the whole generic $G$) from this sequence. However, we can make this sequence fuzzy to result in an element of $\kapman^{L[G]}$. Let $\sim$ be the equivalence relation of eventual coincidence on $({}^{\kappa}2)^{\kapman^{L[G]}}$, i.e. 
$$x\sim y \Leftrightarrow \exists\delta<\kappa\ x\res[\delta, \kappa)= y\res[\delta, \kappa).$$
We call $\langle [c_\beta]_{\sim}\vert\beta<\kappa\rangle$ the \textit{fuzzy sequence}.

\begin{claim}
The fuzzy sequence is an element of $\kapman^{L[G]}$.
\end{claim} 

\begin{proof}
By Lemma \ref{nicedescrip}, it is enough to show that for every $\alpha<\kappa$, $L[G_{\geq I_\alpha}]$ knows of this sequence. But $L[G_{>I_\alpha}]$ contains  the matrix $M_{\geq\alpha}$ and thus the sequence
$$\langle c_\beta\res (\kappa\setminus \alpha)\vert\beta<\kappa\rangle$$
so that $L[G_{\geq\alpha}]$ can compute the relevant sequence of equivalence classes from this parameter.
\end{proof}

Finally, we argue that $\kapman^{L[G]}$ does not contain a choice sequence for the fuzzy sequence\footnote{That is, there is no sequence $\langle x_\beta\mid\beta<\kappa\rangle\in\kapman^{L[G]}$ with $x_\beta\in [c_\beta]_\sim$ for all $\beta<\kappa$.}. Heading toward a contradiction, let us assume that 
$$\langle x_\beta\vert\beta<\kappa\rangle\in\kapman^{L[G]}$$
is such a sequence. $L[G]$ knows about the sequence
$$\langle\delta_{\beta}\vert\beta<\kappa\rangle$$ where $\delta_\beta$ is the least $\delta$ with $x_\beta\res(\kappa\setminus\delta)=c_\beta\res(\kappa\setminus\delta)$. The set of $\lambda<\kappa$ that are closed under the map $\beta\longmapsto\delta_\beta$ is club in $\kappa$. As $\kappa$ is Mahlo in $L[G]$, there is an inaccessible $\alpha=I_\alpha<\kappa$ that is closed under $\beta\longmapsto\delta_\beta$. Now observe that 
$$x_\beta(\alpha)=1\Leftrightarrow c_\beta(\alpha)=1\Leftrightarrow r_\alpha(\beta)=1$$ 
holds for all $\beta<I_\alpha$, so that $r_\alpha\in\kapman^{L[G]}$. But this is impossible as clearly $r_\alpha$ is not fresh.

\end{proof}

Theorem \ref{nochoice} follows.
\medskip\\

\begin{rem}

The only critical property of $L$ that we need to make sure that $\kapman$ is not a model of choice in $L[G]$ is that $L$ has no nontrivial grounds, i.e. $L$ satisfies the ground axiom. $\GCH$ is convenient and implies that no cardinals are collapsed, but it is not necessary. The use of Jensen's covering lemma can also be avoided, as discussed earlier.
\end{rem}

\subsection{The $\omega_1$-mantle}\label{omega1mantlesection}

Up to now, we have focused on the $\kappa$-mantle for strong limit $\kappa$. We will get similar results for the $\omega_1$-mantle. There is some ambiguity in the definition of the $\omega_1$-mantle, depending on whether or not $\omega_1$ is considered as a parameter or as a definition. In the former case, it is the intersections of all grounds $W$ so that $W$ extends to $V$ via a forcing so that $W\models\vert\mathbb P\vert<\omega_1^V$, where in the latter case we would require $W\models\vert\mathbb P\vert <\omega_1^W$. These mantles are in general not equal. To make the distinction clear, we give the latter version the name ``Cohen mantle"\index{Mantle!Cohen} and denote it by $\cohman$. The reason for the name is, of course, that all non-trivial countable forcings are forcing-equivalent to Cohen forcing. 

\begin{lemm}
$\oneman\models\ZF$ and $\cohman\models\ZF$.
\end{lemm}

\begin{proof}
First let us do it for $\cohman$. Clearly, $\cohman$ is closed under the G\" odel operations. It is thus enough to show that $\cohman\cap V_\alpha\in\cohman$ for all $\alpha\in\Ord$. Let $W$ be any Cohen-ground. As Cohen-forcing is homogeneous, $\cohman^V$ is a definable class in $W$. Hence, $\cohman\cap V_\alpha=\cohman\cap V_\alpha^W\in W$. As $W$ was arbitrary, this proves the claim.

Now onto $\oneman$. The above argument shows that all we need to do is show that $\oneman$ is a definable class in all associated grounds. So let $W$ be such a ground. There are two cases. First, assume that $\omega_1^W=\omega_1^V$. Then $W$ extends to $V$ via Cohen forcing, so $\oneman$ is definable in $W$. Next, suppose that $\omega_1^W<\omega_1^V$. This can only happen if $\omega_1^V$ is a successor cardinal in $W$, say $W\models\omega_1^V=\mu^+$. In this case, $W$ extends to $V$ via a forcing of $W$-size $\leq\mu$ and which collapses $\mu$ to be countable. It is well known that in this situation, $W$ extends to $V$ via $\Col(\omega,\mu)$, which is homogeneous as well, so once again, $\oneman$ is a definable class in $W$.
\end{proof}

Once again, choice can fail.

\begin{thm}\label{onemanfail}
Relative to the existence of a Mahlo cardinal, it is consistent that there is no wellorder of $\mathcal{P}(\omega_1^V)^{\oneman}$ in $\oneman$.
\end{thm}

We remark that the Mahlo cardinal is used in a totally different way than in the last section. In the model we will construct, $\omega_1$ will be inaccessible in $\oneman$. Let us once again assume $V=L$ for the rest of the section and let $\kappa$ be Mahlo. Let $\mathbb P$ be the ``${<}\kappa$-support version of $\Col(\omega, {<}\kappa)$", that is
$$\mathbb P=\prod_{\alpha<\kappa}^{<\kappa-\text{support}} \Col(\omega, \alpha).$$
Let us pick a $\mathbb P$-generic filter $G$ over $V$. From now on, $\oneman$ will denote $\oneman^{V[G]}$ and $\cohman$ will denote $\cohman^{V[G]}$.

\begin{prop}\label{axaprop}
Suppose $\QQ$ is a forcing, $\gamma<\lambda$ and $\lambda$ is a cardinal. If $\QQ$ is $\mathrm{AA}(\gamma, \lambda)$ then in $V^\QQ$ there is no surjection from any $\beta<\gamma$ onto $\lambda$.
\end{prop}
\begin{proof}
This is a straightforward adaptation of the proof that Axiom A forcings preserve $\omega_1$.
\end{proof}

 The following lemma is the only significant use of the Mahloness of $\kappa$.

\begin{lemm}
$\PP$ satisfies $\mathrm{AA}(\kappa,\kappa)$. 
\end{lemm}

\begin{proof}
We define $\leq_\alpha$ independent of $\alpha<\kappa$ as the order $\leq^\ast$: Let $p\leq^\ast q$ iff $p\leq q$ and $p\res\supp(q)=q$. The only nontrivial part is showing that for any antichain $A$ and any $p\in\PP$ there is $q\leq^\ast p$ with 
$$\vert\{ a\in A\mid a\Vert q\}\vert<\kappa.$$
Let
$$\PP\res\alpha\coloneqq\{p\in\PP\mid\sup\supp(p)<\alpha\}$$
for all $\alpha<\kappa$. We will proceed to find some $q$ with the desired property. For convenience, we may assume that $A$ is a maximal antichain. As $\kappa$ is Mahlo, there is a regular $\lambda<\kappa$ so that $p\PP\res\lambda$ and any $r\in\PP\res\lambda$ is compatible with some $a\in A\cap\PP\res{\lambda}$. As $V=L$, $\diamondsuit_\lambda$ holds. Thus there is a sequence $\vec d\coloneqq \langle d_\alpha\mid\alpha<\lambda\rangle$ with 
\begin{enumerate}[label=$(\vec d.\roman*)$]
\item $d_\alpha\in\PP_{\leq\alpha}$ and
\item for all $r\in\PP_{\leq\lambda}$ there are stationarily many $\alpha<\lambda$ with $d_\alpha=r\res\alpha$.
\end{enumerate}
Construct a sequence 
$$\langle q_\alpha\mid\alpha<\lambda\rangle$$
of conditions in $\PP\res\lambda$ with $q_\alpha\leq^\ast q_\beta$ for all $\alpha<\beta<\lambda$ as follows: Set $q_0=p$. If $q_\beta$ is defined for all $\beta<\alpha$, let first $q'_\alpha=\bigcup_{\beta<\alpha}q_\beta$ and note that this is a condition. Let $\gamma_\alpha=\sup\supp(q_\alpha')$. Now find $a\in A\cap \PP\res\lambda$ that is compatible with $d_{\gamma_\alpha}$ and let 
$$q_\alpha\coloneqq q_\alpha'^\frown a\res[\gamma_\alpha, \lambda).$$
Finally, set $q=\bigcup_{\alpha<\lambda}q_\alpha$. We have to show that $q$ is compatible with only a few elements of $A$, so suppose $b\in A$ is compatible with $q$. The properties of $\vec d$ guarantee that there is $\alpha<\lambda$ so that 
\begin{enumerate}[label=$(\alpha.\roman*)$]
\item $\gamma_\alpha=\alpha$ and
\item $d_\alpha= b\res\alpha$.
\end{enumerate}
Hence in the construction of $q_{\alpha+1}$ we found some $a\in A\cap\PP\res\lambda$ compatible with $b\res\alpha$ and have $q_{\alpha+1}\res[\alpha,\lambda)\leq a\res[\alpha,\lambda)$. If $a\neq b$, then $a\perp b$ and the incompatibility must lie in the interval $[\alpha, \lambda)$. But then $q_{\alpha+1}$ and $b$ are incompatible as well, contradiction. Thus $b=a$ and it follows that $q$ is compatible with at most $\lambda$-many elements of $A$.
\end{proof}

\begin{cor}\label{omega1col}
We have
\begin{enumerate}[label=$(G.\roman*)$]
\item\label{omega1colcond1} $\omega_1^{L[G]}=\kappa$ and
\item\label{omega1colcond2} if $g\colon \omega\rightarrow\Ord\in L[G]$ then there is some $\alpha<\kappa$ so that $g\in V[G_{\leq\alpha}]$.
\end{enumerate}

\end{cor}

\begin{proof}
To see \ref{omega1colcond1}, note that $\PP$ collapses all cardinals ${<}\kappa$ to $\omega$, so $\omega_1^{L[G]}\geq\kappa$. As $\PP$ satisfies $\mathrm{AA}(\kappa, \kappa)$, there is no surjection from $\omega$ onto $\kappa$ in $L[G]$.\\
Next, let us prove \ref{omega1colcond2}. Let $\dot g\in L$ be a name for $g$. In $L[G]$, find a decreasing sequence of conditions $\langle p_n\mid n<\omega\rangle$ in $G$ so that $p_n$ decides the value of $\dot g(\check n)$ (from the perspective of $L$). Let $\alpha=\sup_{n<\omega}\sup\supp (p_n)$. By \ref{omega1colcond1}, $\alpha<\kappa$. But then $L[G_{\leq\alpha}]$ can compute the whole of $g$. 
\end{proof}

From now on, $\oneman$ denotes $\oneman^{L[G]}$ and $\cohman$ is $\cohman^{L[G]}$.
Let us define an auxiliary model $N=\bigcap_{\alpha<\kappa} L[G_{>\alpha}]$. It is clear that $\oneman\subseteq N$.

Recall the following fact due to Solovay.

\begin{fact}[Solovay, \cite{solovaylebmeasure}]\label{solovayfact}
If $G, H$ are mutually generic filters over $V$ (for any forcings) then $V[G]\cap V[H]=V$.
\end{fact}

\begin{prop}
We have that
\begin{enumerate}[label=$(N.\roman*)$]
\item\label{kapmanNcond1} $N\models\ZF$ and
\item\label{kapmanNcond2} $N\cap\mathcal{P}(\kappa)=\oneman\cap\mathcal{P}(\kappa)=\cohman\cap\mathcal{P}(\kappa)=\{a\subseteq\kappa\mid\forall\beta<\kappa\ a\cap\beta\in V\}$.
\end{enumerate}
\end{prop}

\begin{proof}
First, we will prove \ref{kapmanNcond1}. Once again it is enough to show that $N$ is definable in all models of the form $L[G_{>\alpha}]$ for $\alpha<\kappa$. But this is clear from the definition of $N$.\\
Next, we show \ref{kapmanNcond2}. $\oneman\cap\mathcal{P}(\kappa)\subseteq\cohman\cap\mathcal{P}(\kappa)\subseteq N\cap\mathcal{P}(\kappa)$ is trivial. If $a\in N\cap\mathcal{P}(\kappa)$ and $\beta<\kappa$ then $a\cap\beta\in L[G_{\leq\alpha}]$ for some $\alpha$ by clause \ref{omega1colcond2} of Corollary \ref{omega1col}. As $a\in N$, $a\cap\beta\in L[G_{>\alpha}]$, too. Thus by Fact \ref{solovayfact}
$$a\in L[G_{\leq\alpha}]\cap L[G_{>\alpha}]= L.$$
The final inclusion 
$N\cap\mathcal P(\kappa)\subseteq \oneman\cap\mathcal P(\kappa)$
holds since if $W$ is a ground of $L[G]$ which extends to $L[G]$ via $\mathbb Q$ of size $<\kappa$ then $\mathbb Q$ cannot add a fresh subset of $\kappa$.  
\end{proof}
 
 \begin{proof}[Proof of Theorem \ref{onemanfail}.]
 We will show that in $L[G]$, neither $\oneman$ nor $\cohman$ possess a wellorder of its version of $\mathcal{P}(\kappa)$. In fact, we will show that $N$ does not have such a wellorder, which is enough by \ref{kapmanNcond2} of the above proposition. Once again, let $\sim$ be the equivalence relation on functions $f\colon \kappa\rightarrow\kappa\in N$ of eventual coincidence. For $n<\omega$, let 
$$d_n\colon \kappa\rightarrow\kappa,\ d_n(\alpha)=g_\alpha(n)$$
where $g_\alpha$ is the map $\omega\rightarrow\alpha$ induced by the slice of $G$ generic for $\Col(\omega,\alpha)$. As before, we get that the fuzzy sequence $\langle[d_n]_\sim\mid n<\omega\rangle\in N$. If $N$ had a wellorder of $\mathcal P(\kappa)$, then there would be a choice sequence $\langle x_n\mid n<\omega\rangle\in N$ for the fuzzy sequence. In $L[G]$, one can define the sequence $\langle \delta_n\mid n<\omega\rangle$ where $\delta_n$ is the least point after which $x_n$ and $d_n$ coincide. As $\kappa=\omega_1$ in $L[G]$, the $\delta_n$ are bounded uniformly by some $\delta<\kappa$. But this means that $G_{>\delta}\in N$, a contradiction.  
 \end{proof}

It is natural to conjecture that $N=\cohman=\oneman$, though we do not have a proof of any of these equalities. The problem is that we cannot follow the strategy from Section \ref{kappamahlosection}: $L[G]$ has Cohen-grounds which do not contain any $g_\alpha$ for $\alpha<\kappa$, let alone a tail of the sequence $(g_\alpha)_{\alpha<\kappa}$.

\begin{que}
Is $N=\cohman=\oneman$?
\end{que}

\subsection{The successor of a regular uncountable cardinal case}\label{successorcardinalsection}
We show that, again under $V=L$, for every regular uncountable $\kappa$ there is a forcing extension in which $\succman$ is not a model of $\ZFC$. The upside here is that we do not need any large cardinals at all in our construction, however we pay a price: We do not know whether $\succman$ is a model of $\ZF$ in general.

\begin{thm}\label{succthm}
Assume $V=L$ and suppose $\kappa$ is regular uncountable. Then after forcing with 
$$\PP\coloneqq\prod_{\alpha<\kappa^+}^{<\kappa^+-\text{support}}\Add(\kappa, 1)$$
$\succman$ is not a model of $\ZFC$.
\end{thm}

First, lets do a warm-up with an initial segment of $\PP$. We thank Elliot Glazer for explaining (the nontrivial part of) the following argument to the author.

\begin{lemm}[Elliot Glazer]\label{AAforQ}
If $\kappa$ is regular and $\diamondsuit_\kappa$ holds then 
$$\PP_{\leq\kappa}=\prod_{\alpha<\kappa}^{\text{full support}}\Add(\kappa, 1)$$
satisfies $\mathrm{AA}(\kappa+1, \kappa^+)$.
\end{lemm}

An additional assumption beyond ``$\kappa$ is regular" is necessary here: It is well known that 
$$\prod_{n<\omega}^{\text{full support}}\Add(\omega, 1)$$
collapses $2^\omega$ to $\omega$.
\begin{proof}
We let $p\leq_\alpha q$ if $p\leq q$ and $p\res\alpha=q\res\alpha$. It is easy to see that \ref{AAcond1} and \ref{AAcond3} of Definition \ref{generalaa} hold, so let us show \ref{AAcond2}. Therefore, let $\alpha<\kappa$, $p\in\PP_{\leq\kappa}$ and an antichain $A$ in $\PP_{\leq\kappa}$ be given. As $\diamondsuit_\kappa$ holds, there is a sequence $\langle d_\beta\mid \beta<\kappa\rangle$ with $d_\beta\in\PP_{\leq\beta}$ so that for any $q\in\PP_{\leq\kappa}$ there is some $\beta$ with $q\res\beta=d_\beta$. We will define a sequence $(p_\beta)_{\alpha\leq\beta\leq\kappa}$ inductively so that $p_{\gamma}\leq_\beta p_\beta$ for all $\beta\leq\gamma\leq\kappa$. We put $p_\alpha=p$. At limit stages $\beta$ we let $p_\beta$ be the canonical fusion of $\langle p_\gamma\mid\alpha\leq\gamma<\beta\rangle$. So assume $p_\beta$ is defined. We choose $p_{\beta+1}\leq_\beta p_\beta$ so that, if possible,
$$d_\beta^\frown p_{\beta+1}\leq a$$
for some $a\in A$. Otherwise, we are lazy and set $p_{\beta+1}=p_\beta$.\\ 
Now clearly $q\coloneqq p_\kappa\leq_\alpha p$ and we will show that $q$ is compatible with at most $\kappa$-many conditions in $A$. To see this, suppose $a\in A$ is compatible with $q$. We may find $\beta<\kappa$ so that $d_\beta=a\res\beta$. In the construction of $p_{\beta+1}$ from $p_\beta$, we tried to achieve that $$d_\beta^\frown p_{\beta+1}\res[\beta, \kappa)$$
is below some condition in $A$, which is possible and only possible for $a$. This shows that for any $a\in A$ that is compatible with $q$, there is $\beta<\kappa$ so that $q\res[\beta, \kappa)\leq a\res[\beta,\kappa)$. As $\PP_{\leq\beta}$ has size ${\leq}\kappa$, it follows that there are at most $\kappa$-many such $a\in A$. 
\end{proof}

\begin{cor}
Under the same assumptions as before, $\PP_{\leq\kappa}$ preserves all cardinals $\leq\kappa^+$.
\end{cor}
\begin{proof}
$\PP_{\leq\kappa}$ is ${<}\kappa$-closed and satisfies $\mathrm{AA}(\kappa+1, \kappa^+)$.
\end{proof}

We aim to prove a similar result for $\PP$.
\begin{lemm}\label{succnocollapse}
If $\kappa$ is regular and $\diamondsuit_\kappa$ holds then $\PP$ preserves all cardinals $\leq\kappa^+$. Moreover, if $G$ is $\PP$-generic and $g\colon \kappa\rightarrow\Ord$ is in $V[G]$ then there is $\alpha<\kappa^+$ with $g\in V[G_{\leq\alpha}]$.
\end{lemm}
The argument is similar, but somewhat more complicated. To do so, we introduce a further abstraction of $\AxA(\kappa,\lambda)$.

\begin{defn}
Suppose that $\mathcal P=(P, \preceq)$ is a partial order, $\QQ$ is a forcing, $\kappa<\lambda$ are ordinals. $\QQ$ satisfies \textit{Strategic Axiom A}$(\kappa,\lambda,\mathcal P)$ ($\SAxA(\kappa,\lambda, \mathcal P)$)\index{Strategic Axiom A@Strategic Axiom A$(\kappa,\lambda, \mathcal P)$} if there is a family $\langle \leq_x\mid x\in P\rangle$ of partial orders on $\QQ$ so that
\begin{enumerate}[label=$(\mathrm{SAA}.\roman*)$]
\item\label{saxacond1} $\leq_y\subseteq\leq_x\subseteq\leq_\QQ$ whenever $x\preceq y$ for $x, y\in P$,
\item\label{saxacond2} for any antichain $A\subseteq\QQ$, any $x\in P$ and $p\in\QQ$, there is $q\leq_x p$ with 
$$\vert\{a\in A\mid a\Vert p\}\vert<\lambda$$
and 
\item\label{saxacond3} player II has a winning strategy in the following game we call $\mathcal G(\QQ,\kappa,\mathcal P)$:

\begin{tabular}{c|c|c|c|c|c|c|c|c}
   I  &  $p_0$ & & $p_1$ & & $\dots$ & $p_{\omega}$ & & $\dots$\\
   \hline
   II  & & $x_0$ & & $x_1$ & $\dots$ & & $x_\omega$ & $\dots$
\end{tabular}

The game has length $\kappa$. In an even round $\alpha\cdot 2$,  Player I plays some condition $p_\alpha\in\QQ$ so that $p_\alpha\leq_{x_\beta} p_\beta$ for all $\beta<\alpha$ played so far. In an odd round $\alpha\cdot 2 + 1$, player II plays some $x_\alpha\in P$ with $x_\beta\preceq x_\alpha$ for all $\beta<\alpha$.\\
Player I wins the game iff some player has no legal moves in some round ${<}\kappa$. If the game last all $\kappa$ rounds instead, II wins.
\end{enumerate}
\end{defn}

It is straightforward to generalize Proposition \ref{axaprop}.

\begin{prop}\label{saxaprop}
Suppose $\QQ$ satisfies $\SAxA(\kappa,\lambda, \mathcal P)$. Then in $V^\QQ$, there is no surjection $f\colon \beta\rightarrow\lambda$ for any $\beta<\kappa$.
\end{prop}
\qed

\begin{lemm}\label{SAxAforPlemm}
If $\kappa$ is regular and $\diamondsuit_\kappa$ holds then $\PP$ satisfies 
$$\SAxA(\kappa+1,\kappa^+, \mathcal P_\kappa(\kappa^+))$$
where $\mathcal P_\kappa(\kappa^+)$ is ordered by inclusion.
\end{lemm}
\begin{proof}

For $x\in\mathcal P_\kappa(\kappa^+)$ we will write $p\leq_x q$ if $p\leq q$ and $p\res x= q\res x$. It is clear that \ref{saxacond1} holds.\\
Next, we aim to establish \ref{saxacond3}. We describe a strategy for player II in the relevant game.  We will need to do some additional bookkeeping. Let 
$$h\colon \kappa\rightarrow \kappa\times\kappa$$
be a surjection such that if $h(\beta)=(\alpha, \gamma)$ then $\alpha\leq\beta$. Suppose that $p_\alpha$ is the last condition played by player I and $(x_\beta)_{\beta<\alpha}$ have been played already. In the background, we already have chosen some surjections $s_\beta\colon\kappa\rightarrow\supp(p_\beta)$ for $\beta<\alpha$ and we will adjoin a surjection $s_\alpha:\kappa\rightarrow \supp(p_\alpha)$ to that list. We set 
$$x_\alpha={s_{\gamma_0}(\gamma_1)}\cup\bigcup_{\beta<\alpha} x_\beta$$
where $(\gamma_0,\gamma_1)=h(\alpha)$. As $\kappa$ is regular, $x_\alpha\in \mathcal P_\kappa(\kappa^+)$.

\begin{claim}
Player I does not run out of moves before the game ends.
\end{claim}

\begin{proof}
Suppose we reached round $2\cdot\alpha\leq\kappa$ and let $x=\bigcup_{\beta<\alpha} x_\beta$. We will find a legal play $p_\ast$ for player I. For $\gamma\in\kappa^+\setminus\bigcup_{\beta<\alpha} \supp(p_\beta)$, let $p_\ast(\gamma)$ be trivial. The point is that for $\gamma\in x$, $\langle p_\beta(\gamma)\mid\beta<\alpha\rangle$ stabilizes eventually to some $p_\ast(\gamma)$. If $\alpha=\kappa$, then our bookkeeping made sure that we have 
$$x=\bigcup_{\beta<\kappa}\supp(p_\beta)$$
so that $p_\ast$ is already fully defined and a legal play. If $\alpha<\kappa$ instead, then there are possibly $\gamma\in\bigcup_{\beta<\alpha}\supp(p_\beta)-x$, but then $\langle p_\beta(\gamma)\mid \beta<\alpha\rangle$ is a sequence of length ${<}\kappa$, so we may pick a lower bound $p_\ast(\gamma)\in\Add(\kappa, 1)$ for it.
\end{proof}

It remains to show \ref{saxacond2} and here we will use that $\diamondsuit_\kappa$ holds. Let $\langle d_\beta\mid \beta<\kappa\rangle$ be the ``$\diamondsuit_\kappa$-sequence for $\PP_{\leq\kappa}$" that appeared in the proof of Lemma \ref{AAforQ} and let $A$ be a maximal antichain in $\PP$. Choose $\tau$ to be a winning strategy for player II in $\mathcal G(\PP, \kappa+1,\mathcal P_\kappa(\kappa^+))$ and we will describe a strategy $\sigma$ for player I: Suppose $\alpha\leq\kappa$ and $p_\beta$, $x_\beta$ have already been played for $\beta<\alpha$. This time, we will have picked some surjections $s_\beta\colon\kappa\rightarrow x_\beta$ for $\beta<\alpha$ in the background. Let $x_{<\alpha}\coloneqq\bigcup_{\beta<\alpha}x_\beta$. Then, assuming there is a legal move, pick some $p_\alpha$ so that 

\begin{enumerate}[label=$(p_\alpha.\roman*)$]
    \item $p_\alpha\leq_{x_\beta} p_\beta$ for all $\beta<\alpha$ and
    \item\label{kapmanpalphacond2} if possible, $p_\alpha \res(\kappa^+\setminus x_{<\alpha})\cup e_\alpha\res x_{<\alpha}$ is below a condition in $A$
\end{enumerate}
where $e_\alpha$ is defined by 
$$e_\alpha(s_{\gamma_0}(\gamma_1))=d_\alpha(\gamma)$$
whenever $\gamma<\alpha$ and $h(\gamma)=(\gamma_0, \gamma_1)$ (and $e_\alpha$ is trivial where we did not specify a value)\footnote{$e_\alpha$ may fail to be a function, in which case \ref{kapmanpalphacond2} is void.}. \\
Let $\langle p_\alpha\mid\alpha\leq\kappa\rangle$, $\langle x_\alpha\mid\alpha<\kappa\rangle$ be the sequences of moves played by player I and II in a game where player I follows $\sigma$ and player II follows $\tau$. As $\tau$ is a winning strategy, the sequence must be of length $\kappa+1$. We will show that $q\coloneqq p_\kappa$ is compatible with at most $\kappa$-many elements of $A$. So let $a\in A$ and assume that $q$ is compatible with $a$.

\begin{claim}
There is $\alpha<\kappa$ so that $e_\alpha\in\PP$ and $e_\alpha\res x_{<\alpha}= a\res x_{<\alpha}$.
\end{claim}
\begin{proof}
We define $b\in\PP_{\leq\kappa}$ by $b(\gamma)=a(s_{\gamma_0}(\gamma_1))$ whenever $h(\gamma)=(\gamma_0,\gamma_1)$. Then there is $\alpha<\kappa$ with 
\begin{enumerate}[label=$(\alpha.\roman*)$]
    \item $b\res\alpha= d_\alpha$ and
    \item $x_{<\alpha}=\{s_{\gamma_0}(\gamma_1)\mid\exists\gamma<\alpha\ h(\gamma)=(\gamma_0,\gamma_1)\}$.
\end{enumerate}
It is easy to see now that $\alpha$ is as desired. 
\end{proof}

Thus in round $\alpha\cdot 2$ in the game, player I tried to make sure that
$$a\res x_{<\alpha}\cup p_{\alpha}\res(\kappa^+\setminus x_{<\alpha})$$
is below some condition in $A$. This is possible for $a$, and only for $a$ as $q$ and $a$ are compatible.\\
We have shown that for any $a\in A$ that is compatible with $q$, there is $\alpha<\kappa$ such that $q\res (\kappa^+\setminus x_{<\alpha})\leq a\res (\kappa^+\setminus x_{<\alpha})$. As there are only ${\leq}\kappa$-many $r\in\PP$ with support contained in $x_{<\alpha}$, this implies that there are at most $\kappa$-many such $a\in A$.
\end{proof}

Lemma \ref{succnocollapse} follows from Lemma \ref{SAxAforPlemm} and Proposition \ref{saxaprop} similarly to how we proved Corollary \ref{omega1col}.
\begin{rem}
If additionally $\GCH$ holds at $\kappa^+$ then $\PP$ does not collapse any cardinals at all by a standard $\Delta$-system argument.
\end{rem}

\begin{proof}[Proof of Theorem \ref{succthm}.] Let $G$ be $\PP$-generic over $L$. By Lemma \ref{succnocollapse}, all $L$-cardinals $\leq\kappa^+$ are still cardinals in $L[G]$ (in fact, all cardinals are preserved). Let $N=\bigcap_{\alpha<\kappa^+} L[G_{>\alpha}]$. Using that $N$ is definable in every model of the form $L[G_{>\alpha}]$, it is easy to check that $N$ is a model of $\mathrm{ZF}$. Once again, we call $A\subseteq\kappa^+$ fresh if $A\cap\alpha\in L$ for all $\alpha<\kappa^+$.
\begin{claim}
$\mathcal{P}(\kappa^+)^{
\succman}=\mathcal{P}(\kappa^+)^N=\{A\subseteq\kappa^+\mid A\text{ is fresh}\}^{L[G]}$.
\end{claim} 
\begin{proof}
$\mathcal{P}(\kappa^+)^{\succman}\subseteq\mathcal{P}(\kappa^+)^N$ is trivial. Suppose $A\subseteq\kappa^+, A\in N$. Given $\alpha<\kappa^+$, by Lemma \ref{succnocollapse}, there is $\beta<\kappa^+$ so that $A\cap\alpha\in L[G_{\leq\beta}]$ so that 
$$A\cap\alpha\in L[G_{\leq\beta}]\cap L[G_{>\beta}]=L$$
by Fact \ref{solovayfact}. For the last inclusion assume $A\in L[G]$ is a fresh subset of $\kappa^+$ and $W$ is any $\kappa^+$-ground of $L[G]$. It follows that $W\subseteq L[G]$ satisfies the $\kappa^+$-approximation property so that $A\in W$ as any bounded subset of $A$ is in $L\subseteq W$.
\end{proof}
We will show that there is no wellorder of $\mathcal{P}(\kappa^+)^{\succman}$ in $\succman$. So assume otherwise. Let $\sim$ be the equivalence relation of eventual coincidence on ${}^{\kappa^+}2$ in $N$.  We can realise $G$ as a matrix where the $\alpha$-th row is $\Add(\kappa, 1)$-generic over $L$. Now the columns are in fact $\Add(\kappa^+, 1)$-generic over $L$. Let us write $c_\alpha$ for the $\alpha$-th column ($\alpha<\kappa^+)$ and $d_\beta$ for the $\beta$-th row ($\beta<\kappa$). For any $\alpha<\kappa^+$ we have that $\langle d_\beta\res[\alpha, \kappa^+)\mid\beta<\kappa\rangle\in L[G_{
>\alpha}]$. Thus 
$$\langle [d_\beta]_\sim\mid\beta<\kappa\rangle\in N$$
and by our assumption there must be a choice function, say $\langle x_\beta\mid\beta<\kappa\rangle$, in $N$. In $L[G]$, we can define the sequence $\langle \delta_\beta\mid\beta<\kappa\rangle$, where $\delta_\beta$ is the least point after which $x_\beta$ and $d_\beta$ coincide. As $\kappa^+$ is not collapsed by $\PP$, we can strictly bound all $\delta_\beta$ by some $\delta_\ast<\kappa^+$. But then 
$$\langle x_\beta(\delta_\ast)\mid\beta<\kappa\rangle\in N$$
is $\Add(\kappa, 1)$-generic over $L$, which contradicts that $N$ and $L$ have the same subsets of $\kappa$.
\end{proof}

Note that Fact \ref{stronglimitzffact} does not apply in the situation here, so we may ask:
\begin{que}
Is $\succman$ a model of $\mathrm{ZF}$? Is $\succman=N$?
\end{que}

\section{Conclusion}

There are a number of open questions regarding the interplay between large cardinal properties of $\kappa$ and the $\kappa$-mantle. The following table summarizes what is known as presented in the introduction.\\

\begin{tabular}{l|l}
Large cardinal property of $\kappa$ & Theory of $\kapman$ extends...\\
\hline
 extendible   &  $\ZFC+\mathrm{GA}$\\
measurable & $\ZFC$\\
weakly compact & $\ZF+\kappa\text{-}\mathrm{DC}$\\
inaccessible & $\ZF+(<\kappa, H_{\kappa^+})\text{-choice}$\\
\end{tabular}\\

There is certainly much more to discover here. How optimal are these results? Optimality has only been proven for one of them, namely the first. This is due to Gabriel Goldberg.

\begin{fact}[Goldberg, \cite{gabeextendible}]
Suppose $\kappa$ is an extendible cardinal. Then there is a class forcing extension in which $\kappa$ remains extendible and $\kapman$ is not a $\kappa$-ground. In particular, if $\lambda<\kappa$ and $\mantle_\lambda\models\ZFC$ then $\mantle_\lambda$ has a nontrivial ground.
\end{fact}

The most interesting question seems to be up to when exactly the axiom of choice can fail to hold in $\kapman$. Since this can happen at a Mahlo cardinal, the natural  next test question is whether this is possible at a weakly compact cardinal.

\begin{que}
Suppose that $\kappa$ is weakly compact. Must $\kapman\models\ZFC$?
\end{que}

\bibliographystyle{alpha}
\bibliography{bib.bib}

\end{document}